\documentclass[12pt]{amsart}
\usepackage{amsthm,amsmath,amssymb,tikz-cd}
\usepackage{graphicx}

{\end{list}}
{
   \newtheorem{theorem}{Theorem}[section]
   \newtheorem{proposition}[theorem]{Proposition}
   \newtheorem{lemma}[theorem]{Lemma}

   \newtheorem{corollary}[theorem]{Corollary}

}
{\theoremstyle{definition}

   \newtheorem{example}[theorem]{Example}
   
   \newtheorem{remark}[theorem]{Remark}

}

\newcommand{\RR}{{\mathbb{R}}}

\newcommand{\QQ}{{\mathbb{Q}}}

\newcommand{\ZZ}{{\mathbb{Z}}}

\newcommand{\bfm}{{\mathbf{m}}}
\newcommand{\cA}{{\mathcal A}}

\newcommand{\cF}{{\mathcal F}}

\newcommand{\cL}{{\mathcal L}}

\newcommand{\Ann}{\operatorname{Ann}}

\newcommand{\cLnt}{{\mathcal L^{nt}}}
\newcommand{\Vol}{\operatorname{Vol}}
\newcommand{\MWj}{\operatorname{MW_j}}
\newcommand{\MWd}{\operatorname{MW_d}}
\newcommand{\CH}{\operatorname{CH}}

\newcommand{\Link}{\operatorname{Link}}
\newcommand{\Span}{\operatorname{Span}}

\newcommand{\Star}{\operatorname{Star}}
\newcommand{\Staro}[1]{\operatorname{Star^\circ_{#1}}}

\newcommand{\isomto}{\stackrel{\sim}{\to}}
\newcommand{\isom}{\simeq}

\newcommand{\Po}{Poincar\'e }

\newcommand{\MW}{Minkowski weight }

\newcommand{\setmin}{\,\protect%
\begin{picture}(8,10)\qbezier(1,5.5)(4,4.)(7,2.5)\end{picture}\,}

\newcommand{\bs}{\smallsetminus}

\begin{document}
\title{Mixed volumes of Matroids}

\author{Andy Hsiao, Kalle Karu, Jonathan Yang}
\thanks{This work was partially supported by an NSERC Discovery grant.}
\address{Mathematics Department\\ University of British Columbia \\
  1984 Mathematics Road\\
Vancouver, B.C. Canada V6T 1Z2}
\email{ahsiao@math.ubc.ca, karu@math.ubc.ca, jonathan@math.ubc.ca}

\begin{abstract}
Our aim in this article is to compute the mixed volume of a matroid. We give two computations. The first one is based on the integration formula for complete fans given by Brion. The second computation is a step-by-step method using deletion of elements in the matroid as in \cite{BHM}. 
\end{abstract}

\maketitle


\section{Introduction.}

We consider graded algebras of the form
\[ H = \RR[x_1,\ldots, x_n]/I\]
that satisfy \Po duality. This means that there is an isomorphism $\deg: H^d \isomto \RR$ for some $d\geq 0$, and the bilinear pairing defined by multiplication:
\[ H^i \times H^{d-i} \to H^d \isomto \RR\]
is nondegenrate for all $i\geq 0$. Such algebras are called standard graded, Artininan, Gorenstein, with socle in degree $d$. The mixed volume of this algebra is the linear function on the space of homogeneous degree $d$ polynomials:
\[ \Vol_H: \RR[x_1,\ldots, x_n]_d \to H^d \isomto \RR.\]

One can write the mixed volume itself as a homogeneous polynomial. The dual space of $\RR[x_1,\ldots, x_n]_d$ is $\RR[\partial_1,\ldots, \partial_n]_d$, where we write $\partial_i = \partial/\partial x_i$ for the partial derivatives. Then $\Vol_H$ is a homogeneous degree $d$ partial derivative with constant coefficients.

From now on we will switch the variables $x_i, \partial_i$. We write the algebra as 
\[ H = \RR[\partial_1,\ldots, \partial_n]/I,\]
with the mixed volume a homogeneous degree $d$ polynomial
\[ \Vol_H \in \RR[x_1,\ldots, x_n]_d.\]
The advantage of this convention is that now the ideal $I$ is the annihilator of the polynomial $\Vol_H(x_1,\ldots,x_d)$:
\[ I = \{ D \in \RR[\partial_1,\ldots, \partial_n]\ | \ D(\Vol_H) = 0\}.\]

The Chow ring $\CH(M)$ of a matroid $M$ was defined by Feichtner and Yuzvinsky \cite{FeichtnerYuzvinsky}. It is an algebra as described above, with a canonical degree map $\CH^d(M)\isomto\RR$, defining the mixed volume $\Vol_M$. This mixed volume has been studied extensively. The mixed volume polynomial satisfies various log-concavity properties \cite{AHK, HuhBranden}. Eur \cite{Eur} has given a combinatorial interpretation of the coefficients in the polynomial. In \cite{BES} the mixed volume polynomial was expressed in suitable variables so that all of its coefficients become equal to zero or one. Similar to our approach, in \cite{NOR} the authors study mixed volumes of general fans and then apply the results to the case of matroids.

Our aim in this article is to compute the mixed volume of a matroid in two different ways. The first formula generalizes Brion's integration map for complete fans \cite{Brion}. It gives the mixed volume in the diagonal form -- a linear combination of degree $d$ powers. Such a diagonal form is useful for studying Hodge-Riemann type quadratic forms as these forms will also be diagonalized.

Our second computation of the mixed volume follows the semismall decomposition of the Chow ring in \cite{BHM}. It gives an inductive method for computing the mixed volume from simpler matroids. The mixed volume polynomial reflects the orthogonal decomposition of the Chow ring given in \cite{BHM}.

We will next recall the definition of the Chow ring of a matroid and then describe our main results. We refer to \cite{Oxley} for background on matroids, and we will mostly follow the notation in \cite{BHM} for Chow rings of matroids.

\subsection{Matroids}

Let $M$ be a matroid of rank $d+1$ on the set $E=\{1,2,\ldots,n+1\}$. We will use the definition of a Matroid by its lattice of flats $\cL(M)$. We  assume that the lattice has the minimal flat $\emptyset$ and the maximal flat $E$. All other flats are called nontrivial. We denote the set of nontrivial flats by 
\[ \cLnt(M) = \cL(M)\setmin\{ \emptyset, E\}.\]

To a matroid is associated its Bergman fan $\Delta=\Delta_M$. The fan has rays $\rho_F$ corresponding to nontrivial flats $F\in\cLnt(M)$. Every chain $\cF$ in $\cLnt(M)$:
\[ \cF: F_1< F_2< \cdots < F_m,\]
corresponds to a simplicial cone $\sigma_\cF \in \Delta$ with rays $\rho_{F_1}, \rho_{F_2},\ldots, \rho_{F_m}$. The Bergman fan is a simplicial fan of dimension $d$. It is embedded in $\RR^n$ as follows. Let $\RR^E$ have the standard basis $e_1,\ldots, e_{n+1}$. For a subset $F\subseteq E$, write $e_F = \sum_{i\in F} e_i$. Then $\Delta$ is embedded in $\RR^E/e_E \isom \RR^n$ so that the ray $\rho_F$ is the ray generated by the vector $e_F$ in the quotient. 

The Chow ring of $M$ is the quotient 
\[ \CH(M) = \cA(\Delta)/ \bfm\cA(\Delta).\]
Here $\cA(\Delta)$ is the algebra of piecewise polynomial functions on $\Delta$ (also known as the Stanley-Reisner ring of $\Delta$), and $\bfm$ is the ideal generated by globally linear functions on $\RR^E/e_E$. The Chow ring is known to have a canonical degree map 
\[ \deg: \CH^d(M)\isomto \RR\]
and a nondegenerate \Po duality pairing defined by multiplication.

\begin{example}  \label{ex1}
Consider the matroid $M$ on the set $E=\{1,2,3,4\}$ with nontrivial flats 
\[ 1, 2, 3, 4, 14, 24, 34, 123.\]
This is a representable matroid corresponding to 4 vectors in $\RR^3$, with the first three vectors lying on a plane. The Bergman fan of $M$ is a $2$-dimensional fan in $\RR^3$. It is supported on the union of three half-planes, with each half-plane divided into $3$ maximal cones.

\begin{figure}[ht]
    \centering
    \includegraphics{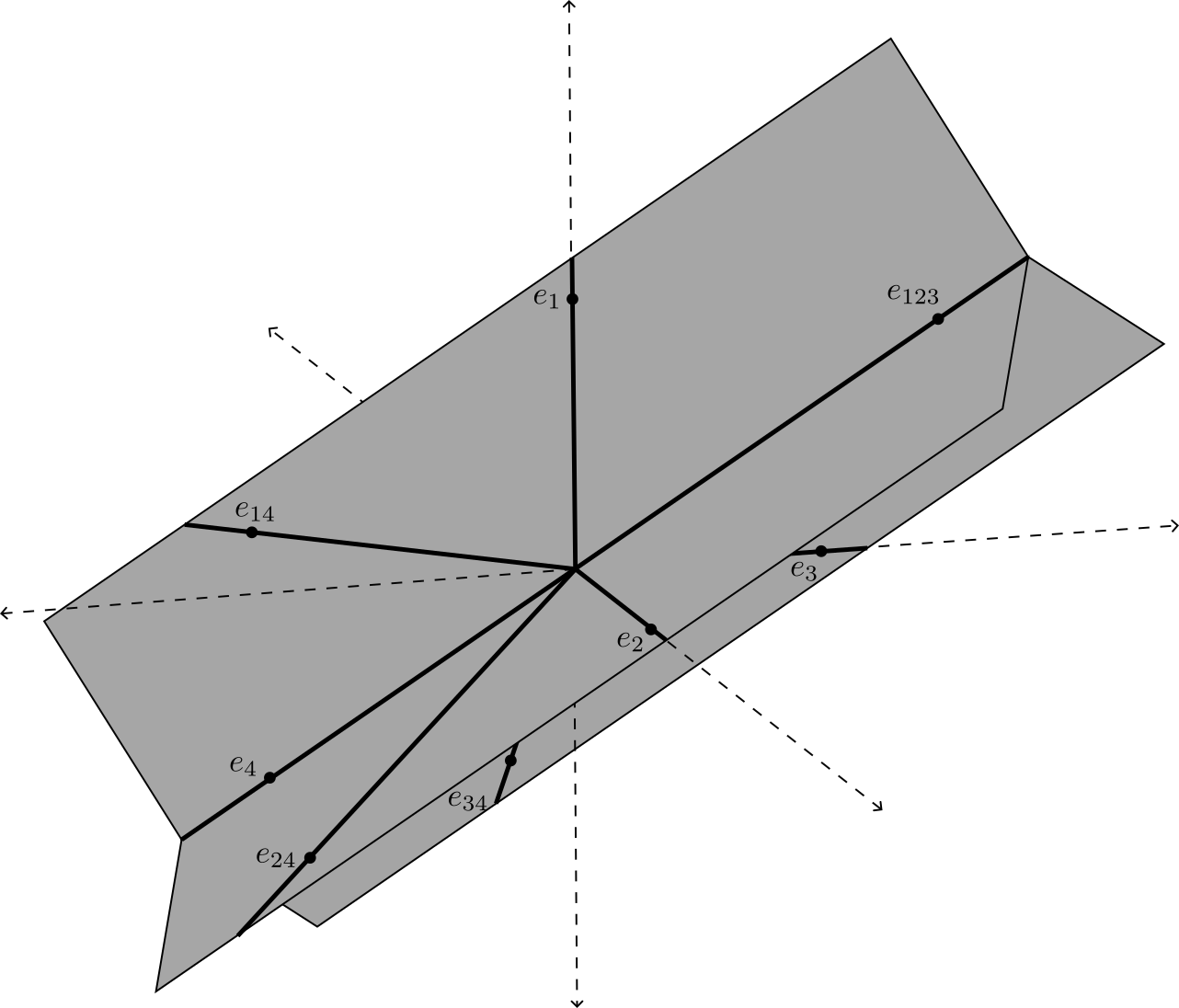}
    \caption{Bergman fan for Example~\ref{ex1}}
    \label{ex1-Bergman-Fan}
\end{figure}

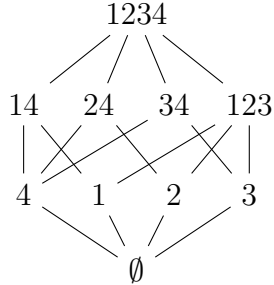
\begin{figure}[ht]
    \centering
    \begin{tikzpicture}
        \node (1234) at (0,0) {\(1234\)};
        
        \node [below of = 1234, yshift = -2mm, xshift = -15mm] (14) {\(14\)};
        \node [below of = 1234, yshift = -2mm, xshift = -5mm] (24) {\(24\)};
        \node [below of = 1234, yshift = -2mm, xshift = 5mm] (34) {\(34\)};
        \node [below of = 1234, yshift = -2mm, xshift = 15mm] (123) {\(123\)};
        
        \draw (1234) to (14);
        \draw (1234) to (24);
        \draw (1234) to (34);
        \draw (1234) to (123);
        
        \node [below of = 24, yshift = -2mm] (1) {\(1\)};
        \node [below of = 34, yshift = -2mm] (2) {\(2\)};
        \node [below of = 123, yshift = -2mm] (3) {\(3\)};
        \node [below of = 14, yshift = -2mm] (4) {\(4\)};
        
        \draw (14) to (1); \draw (14) to (4);
        \draw (24) to (2); \draw (24) to (4);
        \draw (34) to (3); \draw (34) to (4);
        \draw (123) to (1);
        \draw (123) to (2);
        \draw (123) to (3);
        
        \node [below of = 2, xshift = -5mm] (min) {\(\emptyset\)};
        \draw (1) to (min);
        \draw (2) to (min);
        \draw (3) to (min);
        \draw (4) to (min);
    \end{tikzpicture}
    \caption{Lattice of flats for Example~\ref{ex1}}
    \label{fig:enter-label}
\end{figure}
\end{example}

\subsection{First computation}

Our first result is a formula for the mixed volume that generalizes the integration formula of Brion \cite{Brion} for complete fans.

For each $i\in E = \{ 1,2, \ldots, n+1\}$, choose a general vector $v_i \in \RR^{d-1}$ such that $\sum_i v_i = 0$.  For any subset $F\subseteq E$, let  $v_F = \sum_i v_i$.  Consider a maximal cone $\sigma\in\Delta$ corresponding to a  maximal chain of nontrivial flats
\[  F_1 < F_2 < \cdots < F_d ,\]
and let $A_\sigma$ be the $(d-1)\times d$ matrix with columns $v_{F_1}, v_{F_2},  \ldots, v_{F_d}$. We augment this matrix to a $d\times d$ matrix $\hat{A}_\sigma$ by adding the first row of variables $(x_{F_1}, x_{F_2}, \ldots, x_{F_d})$. Then the determinant of $\hat{A}_\sigma$ is a linear form in the variables $x_{F_i}$:
\[ \det \hat{A}_\sigma = a_1 x_{F_1} + a_2 x_{F_2} + \cdots +  a_d  x_{F_d}.\] 
 If $v_i$ are general, then all $d$ coefficients $a_i$ are nonzero. Let $c_\sigma = d! a_1 a_2 \cdots a_d$. Then the degree $d$ form 
 \[  \frac{(\det \hat{A}_\sigma)^d}{c_\sigma} \]
 is scaled so that the monomial $x_{F_1} x_{F_2} \cdots x_{F_d}$ appears with coefficient $1$ in it.
 
 \begin{theorem} \label{thm-MV-first-computation}
 The mixed volume of a matroid $M$ is 
 \[ \Vol_M = \sum_\sigma  \frac{(\det \hat{A}_\sigma)^d}{c_\sigma},\]
where the sum runs over all maximal cones $\sigma=\sigma_\cF \in \Delta$, corresponding to maximal chains $\cF$  in $\cLnt(M)$.  
\end{theorem}

The mixed volume in the theorem does not depend on how we choose the vectors $v_i$. The choice needs to be general in the sense that all $c_\sigma \neq 0$. 

\begin{example}
Consider the matroid in Example~\ref{ex1}. Let the vectors $v_1,\ldots, v_4$ in $\RR^1$ be $-a,-b,-c,a+b+c$. Define the augmented matrix
\[ \hat{A} = \begin{bmatrix} 
x_1 & x_2 & x_3 & x_4& x_{14} & x_{24}& x_{34}& x_{123} \\
-a & -b & -c & a+b+c & b+c & a+c & a+b & -a-b-c
\end{bmatrix}.\]
The matrices $\hat{A}_\sigma$ are the $2\times 2$ submatrices of $\hat{A}$, where we choose the two columns corresponding to the rays of $\sigma$. As an example, if $\sigma$ corresponds to the chain $4 < 14$, then 
\[ \det \hat{A}_\sigma = \det \begin{bmatrix} 
 x_4& x_{14}  \\
 a+b+c & b+c 
\end{bmatrix}
= (b+c) x_4 - (a+b+c) x_{14}.\]
The contribution of the cone $\sigma$ to the mixed volume is 
\[ \frac{(\det \hat{A}_\sigma)^d}{c_\sigma} = \frac{\big((b+c) x_4 - (a+b+c) x_{14}\big)^2}{-2(b+c)(a+b+c)}.\]
The sum of the fractions over all $\sigma$ is a polynomial in $x_F$ with rational coefficients:
\begin{gather*} \Vol_M = -\frac{x_1^2  +x_2^2  +x_3^2  + 2 x_4^2  +x_{14}^2  + x_{24}^2 + x_{34}^2  + x_{123}^2}{2} \\
  + x_4 (x_{14} + x_{24} +  x_{34})  + x_{123}( x_1+ x_2 +x_3) +  x_1 x_{14}   +  x_2 x_{24}  +  x_3 x_{34} .
  \end{gather*}
\end{example}

\subsection{Second computation}
In \cite{BHM} the Chow ring of a matroid was decomposed using the  deletion operation $M\bs i$. We give a formula for computing the mixed volume $\Vol_M$ from $\Vol_{M\bs i}$. This can be used inductively to find the mixed volume of $M$. To simplify notation, we will write the set $\{i\}$ simply as $i$ and use it in expressions such as $F\bs i$, $F\cup i$.

Let $i \in E$. Then $M\bs i$ is a matroid on $E\bs i$ with the set of flats $F\bs i$ over all $F\in \cL(M)$. The element $i$ is called a  coloop if the rank of $M\bs i$ is smaller than the rank of $M$. The Bergman fan of $M\bs i$ is the projection of the Bergman fan of $M$ from the span of $e_i$. This projection takes cones onto cones. In the case where $i$ is a coloop, both $i$ and $E\bs i$ are flats in $M$. The corresponding two rays map to zero by the projection. In the case of  a non-coloop, only the ray $\rho_i$ maps to zero if $i$ is a flat of $M$. If $i$ is not a flat of $M$, then deleting $i$ does not change the lattice of flats or the Chow ring. In this case, let the mixed volume $\Vol_{M\bs i}(x_G)$ have variables indexed by nontrivial flats $G\in \cLnt(M\bs i)$. Then 
\[ \Vol_M = \Vol_{M\bs i}(x_{\overline{G}}),\]
where $\overline{G}$ is the closure of $G$ in $M$.
We will assume from now on that $i$ is a flat of $M$. 

Following \cite{BHM}, define the set $S_i = \{F_1, F_2, \ldots, F_k\}$ consisting of all subsets $F\subset E$ such that $F$ and $F\cup i$ are distinct nontrivial flats of $M$. We choose a total order on $S_i$ that refines the partial order of inclusion. Our first step is to decompose the projection map $\pi: \Delta_M \to \Delta_{M\bs i}$ as a sequence of maps of fans
\begin{equation} \label{eq-fact} 
\pi: \Delta_M=\Delta_0 \to \Delta_1 \to \cdots \to \Delta_k \to \Delta_{M\bs i}.
\end{equation}
Here for $j=1,\ldots, k$, the map $\Delta_{j-1} \to \Delta_j$ is the star subdivision of $\Delta_j$ at the $2$-dimensional cone with rays $\rho_i, \rho_{F_j}$. The subdivision adds the new ray $\rho_{F_j\cup i}$. The last step $\Delta_k \to \Delta_{M\bs i}$ is simply the projection from the span of $e_i$. 

\begin{example}
Consider the matroid in Example~\ref{ex1}. If $i=3$, then $i$ is a non-coloop. The projection of $\Delta_M$ from the span of $e_3$ is again a $2$-dimensional fan. It is the fan of the Boolean matroid $M\bs 3$. In this case $S_3 = \{4\}$. We decompose the projection as a composition of two maps. The first map performs the inverse of a star subdivision and removes the ray $\rho_{14}$. The second map projects the fan.

If we take $i=4$ instead, then $i$ is a coloop. The projection maps the fan $\Delta_M$ to a $1$-dimensional fan in $\RR^2$. In this case $S_4 = \{1,2,3\}.$ Hence the projection can be decomposed as 4 steps, where the first three steps remove the rays $\rho_{14}, \rho_{24}, \rho_{34}$.  The last step contracts both $\rho_4$ and  $\rho_{123}$ to zero.
\end{example}

When $\hat\Delta$ is a subdivision of $\Delta$, then the Chow rings of the two fans are isomorphic in degree $d$. Hence there is a natural way to extend the mixed volume of $\Delta$ to a mixed volume of $\hat\Delta$. In case of a star subdivision at a $2$-dimensional cone this can be described as follows. 
Consider a fan $\Delta$ with rays indexed by a set $L$. Let $\hat{\Delta} \to \Delta$ be the star subdivision map of fans where a cone in $\Delta$ generated by $v_p, v_q$ for some $p,q\in L$ is subdivided by adding a new ray generated by $v_r=v_p+v_q$. Define a derivative $D(p,q)$ in $\RR[[z, \partial_p, \partial_q]]$:
 \begin{gather*} D(p,q) = \frac{\partial_p e^{z \partial_q} - \partial_q e^{z\partial_p}}{\partial_p-\partial_q} \\
 = 1 - \frac{z^2}{2!}\partial_p \partial_q -  \frac{z^3}{3!}\partial_p \partial_q ( \partial_p+\partial_q) - \frac{z^4}{4!}\partial_p \partial_q ( \partial_p^2+ \partial_p \partial_q +\partial_q^2) - \ldots .
 \end{gather*}
Let us also define the derivative $D(p,q,r)$ that acts on polynomials $V\in \RR[x_s]_{s\in L}$  by first applying the derivative $D(p,q)$ and then setting $z=x_r-x_p-x_q$:
\[ D(p,q,r) V = D(p,q) V |_{z=x_r-x_p-x_q}.\]

\begin{theorem} \label{thm-MV-star-subdiv}
Let $\hat{\Delta} \to \Delta$ be the star subdivision at the cone with rays $\rho_p, \rho_q$ that adds a new ray $\rho_r$. Then 
\[ \Vol_{\hat\Delta} = D(p,q,r) \Vol_\Delta.\]
\end{theorem}

We apply the theorem to the first $k$ steps in the factorization (\ref{eq-fact}). Then 
\[ \Vol_{\Delta_{j-1}} = D(i,F_j, F_j\cup i) \Vol_{\Delta_j}.\]

\begin{example} \label{ex1.6}
Let $M$ be the matroid in Example~\ref{ex1}, and let $i=3$. Consider the first map $\Delta_M\to \Delta_1$ in the factorization, where we remove the ray $\rho_{34}$.  Then 
\[ \Vol_M = D(3, 4, 34) \Vol_{\Delta_1} = (1 - \frac{(x_{34}- x_3 -x_4)^2}{2} \partial_3\partial_4) \Vol_{\Delta_1}.\]
\end{example}

Let us now consider the last map in the factorization $\pi: \Delta_k\to \Delta_{M\bs i}$. Assume first that $i$ is not a coloop. Then the Chow rings of the two fans are isomorphic. The mixed volume of $\Delta_k$ is obtained from the mixed volume of $\Delta_{M\bs i}$ by a linear change of variables. 

The map $\pi$ takes takes a ray $\rho_F$ of $\Delta_k$ to the ray $\rho_{F\bs i}$  of $\Delta_{M\bs i}$ if $F\neq i$. Conversely, if $\rho_G$ is a ray of $\Delta_{M\bs i}$, then there is a unique ray $\rho_{\overline{G}}$ of $\Delta_k$ that maps to it. Here $\overline{G}$ is the closure of $G$ in $M$. Thus, $\Vol_{M\bs i}$ is a polynomial in the variables $x_G$ indexed by nontrivial flats $G\in \cLnt(M\bs i)$ and $\Vol_{\Delta_k}$ is a polynomial in the variables $x_i$ and $x_{\overline{G}}$. 

The fan $\Delta_k$ has one new ray $\rho_i$ compared to the fan  $\Delta_{M\bs i}$. One can also find a new linear relation in $\CH(\Delta_k)$ of the form
\[ \partial_i = \sum_{G \in \cLnt(M\bs i)} b_{\overline{G}} \partial_{\overline{G}},\quad b_{\overline{G}} \in\RR.\]
For example, the global linear function $t_i-t_j$ on $\RR^E/ e_E$, where $j\in E \bs i$, gives such a linear relation 
\[  \partial_i = - \sum_{i\in\overline{G}} \partial_{\overline{G}} + \sum_{j\in\overline{G}} \partial_{\overline{G}}.\]

\begin{theorem} \label{MV-non-coloop}
Assume $i$ is not a coloop. Then the mixed volume of $\Delta_k$ is obtained from the mixed volume of $\Delta_{M\bs i}$ by substituting $x_{\overline{G}} + b_{\overline{G}} x_i$ into $x_G$:
\[ \Vol_{\Delta_k} = \Vol_{M\bs i}( x_{\overline{G}} + b_{\overline{G}} x_i).\]
\end{theorem}

Now consider the case where $i$ is a coloop. Then $\Delta_k$ has two extra rays $\rho_i$ and $\rho_{E\bs i}$ compared to $\Delta_{M\bs i}$. These rays correspond to the two new variables $x_i$ and $x_{E \bs i}$ in the mixed volume. The relation now has the form 
\[ \partial_i - \partial_{E\bs i} = \sum_{G \in \cLnt(M\bs i)} b_{\overline{G}} \partial_{\overline{G}},\quad b_{\overline{G}} \in\RR,\]
where we may find $b_{\overline{G}}$ the same way as above. 

If  $g(x)$ is a polynomial, write $\int g(x) dx$ for the antiderivative of $g(x)$ with constant term $0$. 

\begin{theorem} \label{MV-coloop}
Assume that $i$ is a coloop. Then the mixed volume of $\Delta_k$ is
\[ \Vol_{\Delta_k} = \int \Vol_{M\bs i}(x_{\overline{G}} + b_{\overline{G}} x_i) dx_i + \int \Vol_{M\bs i} (x_{\overline{G}} - b_{\overline{G}} x_{E\bs i}) dx_{E\bs i}.\]
\end{theorem}

\begin{example} 
Consider the matroid in Example~\ref{ex1}. Deleting $i=3$ gives a sequence of maps $\Delta_M \to \Delta_1 \to \Delta_{M\bs i}$, where the first map is the star subdivision with new ray $\rho_{34}$, and the second map is projection from $\Span e_3$. Consider the last step $\Delta_1 \to \Delta_{M\bs i}$. For a nontrivial flat $G$ of $M\bs i$, its closure in $M$ is 
\[ \overline{G} = \begin{cases} 123 & \text{ if $G=12$}\\ G & \text{ otherwise}. \end{cases}\]
Using the linear function $t_3-t_2$ gives the relation
\[ \partial_3 = \partial_2+\partial_{24}.\]
Then $\Vol_{\Delta_1}$ is obtained from $\Vol_{M\bs i}$ using the substitution
\begin{align*}
x_{2} & = x_{2} + x_3, \quad x_{24} = x_{24} + x_3,\\
x_{12} & = x_{123}, \quad x_G = x_G \text{ if  $G \neq 2, 24, 12$}.
\end{align*} 
The matroid $M\bs i$ is Boolean and its Bergman fan is a complete fan in $\RR^2$. The mixed volume of this fan can be found by deleting another element, or it can be computed directly as in \cite{KX}. The result is:
\begin{align*} \Vol_{M\bs i} =& -\frac{x_1^2 + x_2^2 + x_4^2 + x_{12}^2+x_{14}^2+x_{24}^2}{2} \\
  &+ x_1(x_{12}+x_{14}) + x_2( x_{12}+x_{24})+ x_4(x_{14}+x_{24}).
  \end{align*}
Performing the substitution as above gives
\begin{align*} \Vol_{\Delta_1} &= -\frac{x_1^2 + (x_2 + x_3)^2 + x_4^2 + x_{123}^2+x_{14}^2+(x_{24}+x_3)^2}{2} \\
  & + x_1(x_{123}+x_{14}) + (x_2+x_3)( x_{123}+x_{24}+x_3)+ x_4(x_{14}+x_{24}+x_3)\\
  &=  \Vol_{M\bs i} +x_3(x_{123}+ x_4).
  \end{align*}
  From Example~\ref{ex1.6} we get
  \[ \Vol_M = (1 - \frac{(x_{34}- x_3 -x_4)^2}{2} \partial_3\partial_4) \Vol_{\Delta_1} = \Vol_{M\bs i} + x_3(x_{123}+x_4) - \frac{(x_{34}- x_3 -x_4)^2}{2} .\]
\end{example}

\begin{example}
    Consider again the matroid in Example~\ref{ex1}. Deleting $i=4$ gives a sequence of maps $\Delta_M \to \Delta_1 \to \Delta_2\to \Delta_3 \to \Delta_{M\bs i}$, where the first three maps are star subdivisions with new rays $\rho_{14}, \rho_{24}, \rho_{34}$, and the last map is projection from $\Span e_4$.

The matroid $M\bs 4$ has mixed volume $x_1+x_2+x_3$. Let us compute the mixed volume of 
$\Delta_3$ from this. The linear function $t_4-t_1$ gives the relation 
\[ \partial_4 - \partial_{123} = \partial_1.\]
Then
\begin{align*} 
\Vol_{\Delta_3} &= \int (x_1 + x_4) + x_2 + x_3 dx_4 + \int (x_1-x_{123}) + x_2 + x_3 dx_{123} \\
 & = x_4(x_1+x_2+x_3 +\frac{x_4}{2}) + x_{123}(x_1+x_2+x_3 - \frac{x_{123}}{2}).
 \end{align*}
Now the mixed volume of $M$ can be found from this by applying the three differential operators
\[ D(i,4,i4) = 1- \frac{(x_{i4}- x_i -x_4)^2}{2}\partial_i\partial_4, \quad i=1,2,3.\]
The result is
\[ \Vol_M = \Vol_{\Delta_3} - \sum_{i=1}^3 \frac{(x_{i4}- x_i -x_4)^2}{2}.\]
As a check, we can find the coefficient of $x_4^2$ to be $\frac{1}{2} - \frac{3}{2} = -1$. 
\end{example}

\subsection{Semismall maps}
In \cite{BHM} it was shown that deleting a non-coloop $i$ from a matroid $M$ gives rise to a semismall morphism of algebraic varieties. This was then used to prove that Hard Lefschetz theorem and Hodge-Riemann bilinear relations for $M\bs i$ imply the same for $M$. 

In the previous section we explained how to break this semismall map into more elementary steps -- the star subdivisions followed by a projection. A star subdivision at a $2$-dimensional cone is a semismall map of fans, which is defined to be a map that takes a cone of dimension $j$ to a cone of dimension at most $2j$.

If $\hat\Delta\to \Delta$ is a star subdivision at a $2$-dimensional cone $\tau\in\Delta$, then one can deduce Hard Lefschetz theorem and  Hodge-Riemann bilinear relations for $\hat{\Delta}$ if we know them for $\Delta$ and for $\Link_\Delta \tau$. 
The Lefschetz element for all fans is $l\in \CH^1(\Delta)$.

In the sequence of star subdivisions $\Delta_{j-1} \to \Delta_j$ in the previous subsection, the fan $\Link_{\Delta_j} \tau$ is a product of Bergman fans and its embedding is also a product embedding. Hard Lefschetz theorem and  Hodge-Riemann bilinear relations for this product follow by induction on dimension. The last step $\Delta_k\to\Delta_{M\bs i}$ induces an isomorphism of Chow rings, hence the two theorems for $\Delta_k$ hold if they hold for $\Delta_{M\bs i}$. One can delete non-coloops $i$ until the matroid becomes Boolean, in which case the theorems follow from the case of simplicial  projective fans. 

\subsection{Semismallness from mixed volume.}
How much of the semismall decomposition can be detected from the mixed volume? The case of one star subdivision can be generalized as follows.

Let $V_0(x_1,\ldots, x_n)$ be a homogeneous degree $d$ polynomial that defines an algebra $H_0$. Consider the algebra $H$ defined by a homogeneous polynomial
\[ V = V_0 + \frac{x_{n+1}}{1!} V_1 + \frac{x_{n+1}^2}{2!} V_2 + \frac{x_{n+1}^3}{3!} V_3+ \ldots,\]
where $V_i\in \RR[x_1,\ldots,x_n]$. Assume that
\begin{enumerate}
    \item $V_1=0$.
    \item $V$ satisfies a differential equation $D(V)=0$, where
    \[ D = \partial_{n+1}^2 + a_1 \partial_{n+1} + a_2\]
    for some $a_i\in \RR[\partial_1,\ldots,\partial_n]$ of degree $i$.
\end{enumerate}
Then $H_0$ is a subalgebra of $H$ and there is an orthogonal decomposition with respect to the \Po pairing:
\[ H = H_0 \stackrel{\perp}\oplus \partial_{n+1} H_0.\]
Moreover, the second summand is
\[ \partial_{n+1} H_0 \isom H_0/\Ann(a_2),\]
where $\Ann a_2 \subseteq H_0$ is the annnihilator of $a_2\in H^2_0$. The \Po pairing on the second summand is given by 
\[ \langle \partial_{n+1} h, \partial_{n+1} g\rangle = V_0(a_2 h g) = V_2(hg).\]

In the case of the star subdivision at a cone with rays $\rho_1, \rho_2$, the derivative is $D=(\partial_1-\partial_{n+1})(\partial_2-\partial_{n+1})$, which vanishes in the Chow ring of the subdivision because the rays $\rho_1, \rho_2$ do not lie in one cone. As explained above, the mixed volume $V$ is obtained from the mixed volume $V_0$ by applying the derivative $D(1,2)$. This derivative has no term of degree $1$ in $z$, hence $V_1=0$.

In general, one cannot deduce  Hard Lefschetz theorem and  Hodge-Riemann bilinear relations for the summand $\partial_{n+1} H_0 = H_0/\Ann a_2$ unless one has more information about this ring or the term $V_2$ in the mixed volume. In the case of star subdivision, this summand is the Chow ring of the link of the cone that is subdivided.

The case of adding one variable $x_{n+1}$ can be generalized to adding variables $x_{n+1}, \ldots, x_{n+m}$. We now consider the mixed volume $V\in \RR[x_1,\ldots, x_{n+m}]$ that satisfies differential equations of the type $D$ above, one for each new variable, and equations $\partial_i \partial_j = 0$ for $n+1\leq i < j\leq n+m$. If this $V$ has no linear terms in $x_i$, $i=n+1,\ldots, n+m$, then it is determined by $V_0$, and there is a direct sum decomposition of $H$ as above with $m+1$ orthogonal summands.

\subsection{Outline of the article.}
Our approach is mostly fan-centric. We prove as much as possible for general fans, and only at the very end specialize to the case of Bergman fans. In the next section we will start with background material about Stanley-Reisner rings of fans. These results are all well-known. We then explain Minkowski weights and mixed volumes for arbitrary fans. In the last section we consider Bergman fans of matroids and prove the theorems listed above.

\section{Fans and Stanley-Reisner rings}

We define all rings over the field $\RR$, although everything also works over $\QQ$.

Let $\Delta$ be a pure simplicial fan of dimension $d$. We fix generators $v_i$, $i=1,\ldots, N$ for the $1$-dimensional cones $\rho_i$ of the fan (the rays of the fan). Denote by $\Delta_j$ the set of $j$-dimensional cones in $\Delta$.

The open star of a cone $\tau\in\Delta$ is
\[ \Staro{\Delta} \tau = \{\sigma\in\Delta| \tau\leq \sigma\}.\]
The closed star $\Star_\Delta\tau$ is the subfan of $\Delta$ generated by the open star.  The link of $\tau$ is the subfan
\[ \Link_\Delta \tau = \{\sigma\in\Star_\Delta \tau| \sigma\cap\tau = 0\}. \]
When $\Delta$ is a fan in a vector space $V$, then we consider $\Star_\Delta \tau$ also in $V$. The subfan $\Link_\Delta \tau$, however, is mapped to the vector space $V/\Span\tau$.

The product of two fans $\Delta_1\times \Delta_2$ has cones $\sigma_1\times\sigma_2$, where $\sigma_1\in\Delta_1$ and $\sigma_2\in\Delta_2$. If $\Delta_1$ is embedded in $V_1$ and $\Delta_2$ is embedded in $V_2$, then $\Delta_1\times\Delta_2$ is embedded in $V_1\times V_2$. We call this embedding the product embedding.

By an abstract fan we mean a fan without an embedding in a vector space. It is a set of cones glued along their faces. This is equivalent to an abstract simplicial complex. A typical situation is that $\Delta$ is a product  $\Delta=\Delta_1\times\Delta_2$ as an abstract fan, but the embedding of $\Delta$ may not be the product embedding.

Let $\cA(\Delta)$ be the ring of $\RR$-valued piecewise polynomial functions on $\Delta$. This ring is graded so that piecewise linear functions have degree $1$. Write $\partial_i$ for the piecewise linear function that takes value $1$ at $v_i$ and $0$ at all other $v_j$. Then we have
\[ \cA(\Delta) = \RR[\partial_1,\ldots,\partial_N]/I,\]
where I is the ideal generated by square-free monomials $\prod_{i\in S} \partial_i$ such that the set $\{ v_i\}_{i\in S}$ does not lie in one cone. The ring $\cA(\Delta)$ is the Stanley-Reisner ring of $\Delta$. It is defined for the abstract fan $\Delta$ and it does not depend on the embedding of $\Delta$ is a vector space.

Let $\theta_1, \theta_2, \ldots, \theta_n \in \cA^1(\Delta)$ be piecewise linear functions. They define a piecewise linear map $\theta: \Delta \to \RR^n$.  If $\theta$ is injective on every cone, we  say that $\theta_1,\ldots, \theta_n$ is a system of linear parameters on $\Delta$. Let $A = \RR[t_1,\ldots, t_n]$ be the ring of global polynomial functions on $\RR^n$. Then $\cA(\Delta)$ is an $A$-module, where $t_i$ acts by multiplication with $\theta_i$.  

Let $\bfm \subseteq A$ be the maximal ideal generated by $t_1,\ldots, t_n$. Define the Chow ring
\[ \CH(\Delta) = \cA(\Delta)/\bfm \cA(\Delta) = \cA(\Delta)/ (\theta_1,\ldots,\theta_n).\]
The Chow ring depends on the linear parameters, but we will omit these from the notation. Given a fixed system of linear parameters $\theta$, we say that $\Delta$ is a fan in $\RR^n$. Then the Chow ring is associated to a fan in $\RR^n$.

For a cone $\sigma$ of dimension $j$, define the monomial $\partial_\sigma = \prod \partial_i \in \cA^j(\Delta)$, where the product runs over the $j$ generators $v_i$ of the cone $\sigma$. Thus, $\partial_0=1$ and $\partial_{\rho_i} = \partial_i$ for a ray $\rho_i$. The following fact is well-known.

\begin{lemma} 
The Chow ring $\CH(\Delta)$ as an $\RR$-vector space is spanned by the square-free monomials $\partial_\sigma$ for  $\sigma\in\Delta$.
\end{lemma}

 \begin{proof}
 We need to show that every monomial in the variables $\partial_i$ in $\cA(\Delta)$ is a linear combination of square-free monomials modulo the ideal $\bfm$.  After re-labeling, consider a monomial $\partial_1^{a_1} \cdots \partial_j^{a_j}$, where all $a_j\geq 1$ and some $a_j$, say $a_1$, is greater than $1$. If $v_1, \ldots, v_j$ do not generate a cone in $\Delta$, then the monomial is zero in $\cA(\Delta)$. Otherwise, there exists a linear combination of $\theta_1,\ldots, \theta_N$ of the form 
 \[  \partial_1 - b_{j+1} \partial_{j+1} - \ldots - b_{N} \partial_{N}.\]
  Then modulo $\bfm$ we can write the monomial as
  \[ (b_{j+1} \partial_{j+1} + \ldots + b_{N} \partial_{N}) \partial_1^{a_1-1}  \partial_2^{a_2} \cdots \partial_j^{a_j}.\]
  If we measure how far a monomial is from being square free using the invariant $\sum_i (a_i-1)$, then every monomial appearing  in the expression has smaller invariant than the original monomial. Now the claim follows by induction.
\end{proof}

The lemma implies that $\CH(\Delta)$ is zero in degrees higher than $d$.

Let us now consider linear relations between the square-free monomials $\partial_\sigma$ in $\CH(\Delta)$. Let $\tau$ be a cone of dimension $j-1$ in $\Delta$, and let $\sigma_i$, $i=1,\ldots, s$, be the $j$-dimensional cones containing $\tau$. We may assume that $v_i\in \sigma_i\bs \tau$. Now consider a linear combination $l$ of $\theta_1, \ldots, \theta_n$ that vanishes on $\tau$,
\[ l = b_1 \partial_1 + \ldots + b_s \partial_s + \ldots.\]
Then the product $l \partial_\tau \in \bfm \cA(\Delta)$ has the form
\begin{equation} \label{eq-extra-rel}
 l \partial_\tau = b_1 \partial_{\sigma_1} + \cdots + b_s \partial_{\sigma_s}. \tag{Rel}
 \end{equation}
 We will see later that such elements, as $\tau$ and $l$ vary, span all relations between $\partial_{\sigma_i}$ in $\CH(\Delta)$. For now we have:
 
 \begin{lemma} \label{lem-extra-rel}
 Consider a system of linear parameters $\theta_1,\ldots, \theta_n$ on $\Delta$ such that the first $d$ of them $\theta_1,\ldots, \theta_d$ also form a system of linear parameters.  Then elements of the form (\ref{eq-extra-rel}) together with $(\theta_1,\ldots,\theta_d)\cA(\Delta)$ span all of $(\theta_1,\ldots,\theta_n)\cA(\Delta)$.
 \end{lemma}
 
 \begin{proof}
 Since the monomials $\partial_\tau$ over all $\tau\in\Delta$ span $\cA(\Delta)/(\theta_1,\ldots,\theta_d)$, it suffices to show that for any $\tau$ and any $\theta_i$, $i=d+1,\ldots, n$, the product $\theta_i \partial_\tau$ lies in the span of elements of the form (\ref{eq-extra-rel}) and $(\theta_1,\ldots,\theta_d)\cA(\Delta)$. Indeed, let $l$ be the sum of $\theta_i$ and a linear combination of $\theta_1,\ldots,\theta_d$, such that $l$ vanishes on $\tau$. Then $l\partial_\tau$ is an element of the form (\ref{eq-extra-rel}), and it is the sum of $\theta_i \partial_\tau$ and an element in $(\theta_1,\ldots,\theta_d)\cA(\Delta)$.
 \end{proof}

\subsection{Star subdivisions} \label{sec-star-subdivisions}
Let $\hat{\Delta}$ be the star subdivision of the fan $\Delta$ at a cone $\tau$. We extend the linear parameters $\theta_1,\ldots,\theta_n$ from $\Delta$ to $\hat{\Delta}$ so that the new ray is generated by $v_0 = v_1+\ldots+ v_j$ in $\RR^n$, where $v_1,\ldots, v_j$ are the generators of $\tau$. Pullback of piecewise polynomial functions defines an injective homomorphism of $A$-algebras
\[ \cA(\Delta) \hookrightarrow \cA(\hat\Delta)\]
and a morphism of Chow rings $\CH(\Delta) \to \CH(\hat\Delta)$. Let us write $\hat{\partial}_i$, $i=0,1 ,\ldots,N$, for the generators of $\cA(\hat\Delta)$.

\begin{lemma} \label{lem-star-subdivisions}
There is a direct sum decomposition of $A$-modules
\[ \cA(\hat\Delta) = \cA(\Delta) \oplus \big[ \hat\partial_{0} \cA(\Delta) \oplus \hat\partial_{0}^2 \cA(\Delta) \oplus \cdots \oplus \hat\partial_{0}^{j-1} \cA(\Delta)\big].\]
Modulo the ideal $\bfm$, this gives a direct sum of vector spaces 
\[ \CH(\hat\Delta) = \CH(\Delta) \stackrel{\perp}{\oplus} \big[ \hat\partial_{0} \CH(\Delta) \oplus \hat\partial_{0}^2 \CH(\Delta) \oplus \cdots \oplus \hat\partial_{0}^{j-1} \CH(\Delta)\big],\]
where the two summands are orthogonal in the following sense: the product of any two elements $f$ and $g$ from the two summands, such that $\deg{f}=p$ and $\deg{g}=d-p$ for some $p$, is zero in $\CH^d(\hat{\Delta})$.
\end{lemma}

\begin{proof}
This result is a special case of the decomposition theorem \cite{BBFK, BL1}. We only sketch the proof here.

The pullback map takes $\partial_i$ to $\hat{\partial}_i + \hat\partial_{0}$ for $i=1,\ldots,j$, and to $\hat\partial_i$ for any other $i$. Since $v_1,\ldots, v_j$ do not generate a cone in $\hat\Delta$, the following equality holds in $\cA(\hat\Delta)$:
\[ 0= \hat\partial_1 \hat\partial_2\cdots \hat\partial_j= (\partial_1- \hat\partial_{0}) (\partial_2 - \hat\partial_{0}) \cdots (\partial_j - \hat\partial_{0}).\]
This gives an integral relation of degree $j$ for $\hat\partial_{0}$ over $\cA(\Delta)$. It shows that $\cA(\hat\Delta)$ is a sum subspaces as in the statement of the lemma, but possibly not a direct sum. To prove that there are no relations between the summands, it suffices to restrict all piecewise polynomial functions to the subfan $\Star\tau \subseteq \Delta$. Let's assume $\Delta=\Star\tau$. Then as abstract fans,
\[ \Delta = [\tau] \times \Link\tau, \quad \hat\Delta = [\hat\tau] \times \Link\tau,\]
where $[\tau]$ is the fan consisting of $\tau$ and all its faces, and $[\hat\tau]$ is its star subdivision. From this we get
\[ \cA(\Delta) = \cA([\tau]) \otimes \cA(\Link\tau), \quad \cA(\hat\Delta) = \cA([\hat\tau]) \otimes \cA(\Link\tau).\]
Now  $\cA([\hat\tau])$ is equal to  $\cA([\tau])[\partial_{0}]$ modulo the single relation given above.

In the direct sum decomposition of $\CH(\hat\Delta)$, the second summand is a module over the first summand. It follows that any product $fg$ as in the statement of the lemma lies in the second summand. We claim that the second summand is zero in degree $d$.
For each $i=1,\ldots, j-1$, we have an isomorphism
\[ \hat\partial_{0}^i \CH(\Delta) \isom \hat\partial_{0}^i \CH( \Link \tau).\]
Since $\Link\tau$ has dimension $d-j$, all these summands are zero in degree $d$.  
\end{proof}

Given an arbitrary simplicial fan $\hat\Delta$ with a ray $\rho_0$, we can ask: when is the fan $\hat\Delta$ the star subdivision of a simplicial fan $\Delta$ at a $j$-dimensional cone $\tau$ so that the new ray is $\rho_0$? If $\hat\Delta$ is a such a subdivision, then as abstract fans,
\[ \Link_{\hat{\Delta}} \rho_0 = \Pi^{j-1} \times L, \]
where $\Pi^{j-1}$ is the boundary fan of $\tau$ and $L$ is the link of $\tau$ in $\Delta$. Conversely, if the link of $\rho_0$ is such a product for some subfan $L\subseteq \hat\Delta$, then $\hat\Delta$ as an abstract fan is a star subdivision. To get $\Delta$, we replace the subfan 
\[ \Star_{\hat{\Delta}} \rho_0 = [\rho_0] \times \Pi^{j-1} \times L,\]
with the fan 
\[ \Star_\Delta\tau = [\tau] \times L.\]

\begin{remark} \label{rem-edge-contration}
The inverse operation of star subdivision (star assembly) can also be viewed as an edge contraction. Fix a vertex, say $v_1$ of the cone that is subdivided. Then $\Delta$ is obtained from $\hat{\Delta}$ by contracting the $2$-dimensional cone generated by $v_0, v_1$ to the ray generated by $v_1$. This operation also contracts every cone that contains $v_0, v_1$ to a cone of one smaller dimension. 
\end{remark}

 \section{Minkowski weights}
 
 Consider a fan $\Delta$ of dimension $d$ in $\RR^n$. A \MW of degree $j$ on $\Delta$ is a vector $(w_\sigma)$, indexed by $j$-dimensional cones $\sigma\in\Delta$, that satisfies a balancing condition for every $(j-1)$-dimensional cone $\tau\in\Delta$. Let $\sigma_1,\ldots, \sigma_s \in \Delta_j$ be the cones containing $\tau$ as a face. We may assume that $\sigma_i\bs \tau$ contains $v_i$. Then the balancing condition is that the vector
 \[ \sum_{i=1}^s w_{\sigma_i} v_i \in \RR^n\]
 lies in the span of $\tau$.
 
 Minkowski weights of degree $j$ form a vector space that we denote by $\MWj(\Delta)$. The space of Minkowski weights on $\Delta$ is dual to $\CH^j(\Delta)$:
 
 \begin{theorem} \label{thm-MW}
 There exists an isomorphism
 \[ \Phi: \big( \CH^j(\Delta) \big)^* \to \MWj(\Delta)\]
 that maps a linear function $\phi: \CH^j(\Delta) \to \RR$ to the vector  $(w_\sigma = \phi(\partial_\sigma))$.
 \end{theorem}
 
 This theorem is not new, see for example \cite{FultonSturmfels, AHK}. We give a proof for it because the construction of the inverse of $\Phi$ in degree $d$ defines for every $w\in \MWd(\Delta)$ a linear map $\phi_w: \CH^d(\Delta) \to \RR$. We use this to associate to every $w\in \MWd(\Delta)$ the mixed volume
 \[ \Vol_w : \RR[\partial_1,\ldots,\partial_N] \to \CH^d(\Delta) \stackrel{\phi_w}{\longrightarrow} \RR.\]

It suffices to prove the theorem for $j=d$. The case $j<d$ follows by restricting to the $j$-skeleton of $\Delta$.
When $\Delta$ is a complete fan in $\RR^d$, then $\MWd(\Delta)$ is $1$-dimensional, spanned by $(w_\sigma = 1/\det \sigma)$. Here $\det \sigma = |\det [v_1, \ldots, v_d]|$, where $v_1,\ldots, v_d$ are the generators of $\sigma$.

Let us start by proving that the map $\Phi$ of the theorem has image in $\MWj(\Delta)$.

\begin{lemma} \label{lem-MW-balance}
For a linear map $\phi: \cA^j(\Delta) \to \RR$, the vector $(w_\sigma=\phi(\sigma))$ lies in $\MWj(\Delta)$ if and only if $\phi$ maps degree $j$ relations of the form (\ref{eq-extra-rel}) to zero.

In particular, if $\phi$ descends to a map $\phi: \CH^j(\Delta) \to \RR$,  then the vector  $(w_\sigma = \phi(\partial_\sigma))$ lies in $\MWj(\Delta)$.
\end{lemma}

\begin{proof}
Let $\tau$ be a cone of dimension $j-1$, and let $l$ be a global linear function on $\RR^n$ that vanishes on $\tau$. Then $(w_\sigma)$ satisfies the balancing condition at $\tau$ if each such $l$ vanishes on $v= \sum_{i=1}^s w_{\sigma_i} v_i$. Note that $l \partial_\tau$ defines an element as in equation (\ref{eq-extra-rel}):
\[ l\partial_\tau = b_1 \partial_{\sigma_1} + \cdots + b_s \partial_{\sigma_s} \in \bfm \cA(\Delta).\]
Since $b_i = l(v_i)$, the function $\phi$ applied to this expression is equal to to  $l(v)$:
\[ \phi(l\partial_\tau) = l(v_1) w_{\sigma_1} + \cdots +  l(v_s) w_{\sigma_s}.\]
Hence, $(w_\sigma)$ satisfies the balancing condition at $\tau$ if and only if $\phi$ maps degree $j$ relations (\ref{eq-extra-rel}) to zero.

The second claim follows because in this case $\phi$ vanishes on $\bfm \cA(\Delta)$, which contains all relations (\ref{eq-extra-rel}).
\end{proof}

\begin{remark}
Theorem~\ref{thm-MW} and Lemma~\ref{lem-MW-balance} together imply that $\CH(\Delta)$ is spanned with square free monomials $\partial_\sigma$ for $\sigma\in\Delta$, and all relations between these monomials are generated by   (\ref{eq-extra-rel}). 
\end{remark}

Next we find the inverse of $\Phi$ in the case where $j=d=n$. This construction generalizes the integration formula of Brion \cite{Brion}.
If $\Delta$ is a fan in $\RR^d$, then for a piecewise polynomial function $f=(f_\sigma)_{\sigma\in\Delta_d}$ we may view every piece $f_\sigma$ as a polynomial in $\RR[t_1,\ldots,t_d]$. Let $\chi_\sigma$ be the piecewise polynomial function $\partial_\sigma$ restricted to $\sigma$:
 \[ \chi_\sigma = \partial_{\sigma,\sigma} \in \RR[t_1,\ldots,t_d].\]

\begin{lemma} \label{lem-Brion}
Let $\Delta$ be a $d$-dimensional fan in $\RR^d$, and let $w=(w_\sigma) \in \MWd(\Delta)$. Consider the linear map
\begin{align} 
\phi_w: \cA(\Delta) & \to \RR(t_1, \ldots, t_d)  \nonumber \\
 f & \mapsto  \sum_{\sigma \in \Delta_{d}} w_\sigma \frac{f_\sigma}{\chi_\sigma}. \label{eq-Brion0}
 \end{align} 
Then the image of $\phi_w$ lies in $\RR[t_1, \ldots, t_d]$ and the induced map
\[ \phi_w: \cA(\Delta)  \to \RR[t_1, \ldots, t_d]\]
is a degree $-d$ homomorphism of graded $\RR[t_1,\ldots, t_d]$ modules. The map $\phi_w$ in degree $d$ descends to a linear map
\begin{equation*} 
 \phi_w: H^d(\Delta) \to \RR
 \end{equation*}
 such that $\phi_w(\partial_\sigma) = w_\sigma$.
\end{lemma}

\begin{proof}
    We first show that the image of $\phi_w$ lies in $\RR[t_1, \dots, t_d]$.  Since each $\chi_\sigma$ is a product of linear functions vanishing on hyperplanes spanned by the facets of $\sigma$, it suffices to check that $\sum_{\sigma \in \Delta_{n}} w_\sigma \frac{f_\sigma}{\chi_\sigma}$ has no  poles along such hyperplanes.
    
    Fix $f \in \cA(\Delta)$ and a cone $\tau \in \Delta_{d-1}$. Let $S= \{\sigma_1,\ldots,\sigma_s\}$ be the set of maximal cones containing $\tau$. Label the generators of $\tau$ as $v_1, \dots, v_{d-1}$, and the generators of $\sigma \in S$ as $v_1, \dots, v_{d-1}, v_d^\sigma$. Let $t$ be a linear function defining the hyperplane spanned by $\tau$. It then suffices to show
    \[
        t \sum_{\sigma\in S} w_\sigma \frac{f_\sigma}{\chi_\sigma}
    \] 
    vanishes on the hyperplane. On the hyperplane we have $f_{\sigma} = f_{\sigma'}$ for any two $\sigma,\sigma' \in S$, so define $f_\tau = f_\sigma$. For each $\sigma \in S$, $\chi_\sigma$ is the piecewise polynomial function  $\partial_1 \cdots \partial_{d-1} \partial_d^\sigma$ restricted to $\sigma$. Each piecewise linear functions $\partial_i$ for $i=1\ldots, d-1$ restricts to the same linear function on the hyperplane, independent of $\sigma$. The piecewise linear function $\partial_d^\sigma$ restricted to $\sigma$ is a linear function vanishing on the hyperplane, $\partial_d^\sigma|_\sigma = t/t(v_d^\sigma)$. So together, 
    \begin{align*}
        t\sum_{\sigma \in S} w_\sigma \frac{f_\sigma}{\chi_\sigma} &= t\frac{f_\tau}{\partial_1\dots\partial_{d-1}}\sum_{\sigma \in S} w_\sigma \frac{t(v_n^\sigma)}{t} = \frac{f_\tau}{\partial_1\dots\partial_{d-1}} t\left( \sum_{\sigma \in S}w_\sigma v_n^\sigma \right) = 0 
    \end{align*} by the balancing condition of Minkowski weights. 

    To see that $\phi_w$ descends to $\CH^d(\Delta) \to \RR$, recall $\CH^d(\Delta) = \cA^d/(t_1, \dots, t_d)\cA^{d-1}(\Delta)$, and we observe that $\cA^{d-1}(\Delta)$ maps to zero as $\phi_w$ is a degree $-d$ map. Since each $\partial_\sigma$ is non-zero only on $\sigma$, we indeed have $\phi_w(\partial_\sigma)=w_\sigma$.
\end{proof}

The previous lemma can be extended to the case $j=d<n$.

\begin{lemma} \label{lem-Brion-n}
Let $\Delta$ be a $d$-dimensional fan with a system of linear parameters $\theta_1,\ldots, \theta_n$, and let $w=(w_\sigma)$ be in $\MWd(\Delta)$. Assume that $\theta_1, \ldots, \theta_d$ also define a linear system of parameters on $\Delta$. Then, viewing $\Delta$ as a fan in $\RR^d$,  $w$ is again a mixed volume, and the map $\phi_w: \cA^d(\Delta) \to \RR$ defined in Lemma~\ref{lem-Brion}  descends to a linear map
\[ \phi_w: \cA^d(\Delta)/(\theta_1,\ldots,\theta_n)\cA^{d-1}(\Delta) \to \RR,\]
such that $\phi(\partial_\sigma)=w_\sigma$.
\end{lemma}

\begin{proof}
The map $\Delta \to \RR^d$ is the composition of the map to $\Delta\to\RR^n$ with a linear projection $\RR^n\to\RR^d$. The projection  maps a balancing relation to a balancing relation, hence $w$ is a \MW on the fan $\Delta$ in $\RR^d$.

Define the map $\phi_w: \cA(\Delta) \to \RR$ using Lemma~\ref{lem-Brion} when viewing $\Delta$ as a fan in $\RR^d$. This map satisfies $\phi_w(\partial_\sigma) = w_\sigma$ for $\sigma\in\Delta_d$. By Lemma~\ref{lem-extra-rel}, 
 $(\theta_1,\ldots, \theta_n)\cA^{d-1}(\Delta)$ is spanned by  $(\theta_1,\ldots, \theta_d)\cA^{d-1}(\Delta)$ and elements of the form (\ref{eq-extra-rel}). The map $\phi_w$ vanishes on the first subspace by Lemma~\ref{lem-Brion}, and on the elements of the form (\ref{eq-extra-rel}) by Lemma~\ref{lem-MW-balance}.
\end{proof}

We are now ready to prove the theorem.

\begin{proof}[Proof of Theorem~\ref{thm-MW}]
Assume that $j=d$. Let $\Delta$ have a linear system of parameters $\theta_1,\ldots,\theta_n$. After a linear change of coordinates we may assume that the first $d$ of them $\theta_1,\ldots,\theta_d$ also form a linear system of parameters. 
The previous lemma now defines a map $\phi_w: \CH^d(\Delta) \to \RR$ such that $\phi(\sigma)=w_\sigma$.
The map $w\mapsto \phi_w$ is the inverse of $\Phi$.
\end{proof}

\section{Mixed volumes}

Let $\Delta$ be a $d$-dimensional fan in $\RR^n$, and let $w\in\MWd(\Delta)$. We call the linear map
\[ \Vol_w: \RR[\partial_1,\ldots,\partial_N]_d \to \cA^d(\Delta) \stackrel{\phi_w}{\longrightarrow} \RR\]
the mixed volume associated to $w$. Let $\partial_i=\frac{\partial}{\partial x_i}$ be the partial derivatives. Then there exists a homogeneous degree $d$ polynomial in variables $x_1,\ldots, x_N$ such that the mixed volume is defined by applying a degree $d$ partial derivative to this polynomial. We denote the polynomial also $\Vol_w(x_1,\ldots, x_N)$.

\subsection{Mixed volume as an evaluation map}

Let $\Delta$ be a fan in $\RR^d$. In this case Lemma~\ref{lem-Brion} provides a formula for the  map $\phi_w: \cA^d{\Delta} \to \RR$:
\begin{equation} \label{eq-Brion} 
\phi_w(f) =  \sum_{\sigma \in \Delta_{n}} w_\sigma \frac{f_\sigma}{\chi_\sigma}.
\end{equation}
When $\Delta$ is a fan in $\RR^n$, then Lemma~\ref{lem-Brion-n} says that $\phi_w$ can be defined by first projecting the fan to $\RR^d$. The map $\phi_w$ does not depend on the projection as long as it is general enough. We may therefore restrict to the case $n=d$ for the rest of this subsection.

We fix a general vector $v_0\in\RR^d$ and evaluate the rational function on the right hand side of $(\ref{eq-Brion})$ at $v_0$:
\[\phi_w(f) =  \sum_{\sigma \in \Delta_{n}} w_\sigma \frac{f_\sigma(v_0)}{\chi_\sigma(v_0)}.\]
The result is a constant that does not depend on which $v_0$ is used.
We will next find the value of $f_\sigma/\chi_\sigma$ at $v_0$. Let us start by evaluating the restriction of $\partial_i$ to $\sigma$ at $v_0$.

Let $\sigma$ be a cone of dimension $d$, and without loss of generality assume that $\sigma$ is generated by $v_1,\ldots, v_d$. Let us record the coordinates of the vectors $v_0,v_1,\ldots, v_d$ in a $d\times(d+1)$ matrix with $v_i$ as columns:
\[ A_\sigma = [a_{i,j}] =  [v_0, v_1, v_2, \ldots, v_d]. \]
There exists an isomorphism of $\RR$-algebras
\begin{align*}
\alpha_\sigma: \RR[t_1,\ldots, t_d] & \to \RR[\partial_1,\ldots, \partial_d], \\
t_i &\mapsto a_{i,1} \partial_1 + a_{i,2} \partial_2 + \cdots + a_{i,d} \partial_d.
\end{align*}
The inverse of $\alpha_\sigma$ expresses a piecewise polynomial function restricted to $\sigma$ as a polynomial in $t_i$.

Let $X_i^\sigma\in\RR$ for $i=0,\ldots, d$ be $(-1)^i$ times the determinant of the matrix $A_\sigma$ with its column containing $v_i$ removed.

\begin{lemma} \label{lem-eval}
The polynomial $\alpha^{-1}_\sigma(\partial_i)$ evaluated at $v_0$ is $-X_i^\sigma/X_0^\sigma$.
\end{lemma}

\begin{proof} 
Suppose $v_0 = v_j$ for some $j=1,\ldots, d$. Then 
\[ X_i^\sigma = \begin{cases} 
 -X_0^\sigma & \text{if $i=j$,} \\ 
0 & \text{otherwise}
\end{cases}\]
 In this case $-X_i^\sigma/X_0^\sigma$  is the correct value of $\partial_i$ at $v_j$. A general $v_0$ is a linear combination of $v_1,\ldots, v_d$. The value of $\partial_i$ at $v$ is linear in $v$, hence is also correct in general.
\end{proof}

\begin{theorem} \label{thm-MV-eval}
The mixed volume $\Vol_w: \RR[\partial_1,\ldots,\partial_N]_d \to \RR$ is 
\[ \Vol_w f(\partial_1,\ldots, \partial_N) = \sum_{\sigma \in \Delta_{n}} w_\sigma \frac{f(X_1^\sigma, X_2^\sigma,\ldots, X_d^\sigma)}{X_1^\sigma X_2^\sigma\cdots X_d^\sigma}.\]
Here in each summand we evaluate $f$ by setting $\partial_i$ equal to $X_i^\sigma$ if $v_i$ is a generator of $\sigma$, and zero otherwise.
\end{theorem}

\begin{proof} 
We evaluate the sum (\ref{eq-Brion}) at $v_0$ by substituting $-X_i^\sigma/X_0^\sigma$ in $\partial_i$. Since both numerator and denominator are homogeneous of degree $d$, we may cancel the negative signs and $(X_0^\sigma)^d$.
\end{proof}

\begin{corollary} \label{cor-MV-derivative}
The mixed volume as a polynomial is
\[ \Vol_w(x_1,\ldots,x_N) = \sum_{\sigma \in \Delta_{d}} \frac{w_\sigma}{d!}  \frac{(X_1^\sigma x_1^\sigma+  X_2^\sigma x_2^\sigma+ \cdots + X_d^\sigma x_d^\sigma)^d}{X_1^\sigma X_2^\sigma\cdots X_d^\sigma}.\]
Here $x_1^\sigma,\ldots, x_d^\sigma$ are the variables $x_i$ that correspond to the generators $v_i$ of $\sigma$.
\end{corollary}

\begin{proof}
If $\partial_1^{a_1} \cdots \partial_d^{a_d}$ is a degree $d$ partial derivative with $a_i\in \ZZ_{\geq 0}$, and $X_1, \ldots, X_n \in \RR$ are constants, then 
\[  \frac{1}{d!} \partial_1^{a_1} \cdots \partial_d^{a_d} (X_1 x_1+  X_2 x_2+ \cdots + X_d x_d)^d = X_1^{a_1} \cdots X_d^{a_d}.\]
This shows that applying  a derivative $f$ to the polynomial in $x_i$ is the same as evaluating $f$ at $X_i$.
\end{proof}

The previous corollary can be restated as follows. Let $\hat{A}_\sigma$ be the $(d+1)\times(d+1)$ matrix where we add to $A_\sigma$ the first row $(0,x^\sigma_1, \ldots, x^\sigma_d)$. Then
\[ \det(\hat{A}_\sigma) = X_1^\sigma x_1^\sigma+  X_2^\sigma x_2^\sigma+ \cdots + X_d^\sigma x_d^\sigma.\]
Let us also write $c_\sigma$ for the coefficient of $x^\sigma_1\cdots x^\sigma_d$ in the $d$-th power of this linear form. Then
\[ c_\sigma = d! X^\sigma_1\cdots X^\sigma_d.\]
Corollary~\ref{cor-MV-derivative} can now be stated as:

\begin{corollary} \label{cor-MV-derivative2}
The mixed volume as a polynomial is
\[ \Vol_w(x_1,\ldots,x_N) = \sum_{\sigma \in \Delta_{d}} w_\sigma  \frac{\det(\hat{A}_\sigma)^d}{c_\sigma}.\]
\end{corollary}

\subsection{Decomposition of mixed volume.}

When the \MW $w$ is a linear combination of Minkowski weights, $w= \sum_i b_i w_i$, then the mixed volume is also a linear combination of mixed volumes $\Vol_w=\sum_i b_i \Vol_{w_i}$. We are interested in the case where each $w_i$ is supported on a subfan  $\Delta_i\subseteq \Delta$. In this case we can view $\Vol_{w_i}$ as a mixed volume on $\Delta_i$.  We describe one such decomposition that works for all fans.

Let $\Delta$ be a fan in $\RR^d$. Construct a new fan $\tilde{\Delta}$ as follows. Choose a new ray generated by a sufficiently general vector $v_0$, and for every cone $\tau\in \Delta_{d-1}$, let $\sigma_\tau$ be the $d$-dimensional cone generated by $\tau$ and $v_0$. These cones $\sigma_\tau$ together with cones in $\Delta$ generate $\tilde{\Delta}$. 

For each $\sigma\in\Delta_d$, let $\Delta_\sigma$ be the subfan of $\tilde{\Delta}$ with maximal cones $\sigma$ and all $\sigma_\tau$ for $\tau\leq \sigma$. This fan is the boundary fan of a simplicial $(d+1)$-cone. Note that any \MW $u$ on a subfan such as $\Delta_\sigma$ or $\Delta$ can be viewed as a \MW on $\tilde{\Delta}$, extending it by zero outside the subfan.

\begin{theorem} 
For  each $\sigma\in\Delta_d$, the weight $w_\sigma$ on $\sigma$ extends to a unique \MW $u^\sigma$ on $\Delta_\sigma$. Then 
\[ w =\sum_{\sigma\in \Delta_d} u^\sigma.\]
\end{theorem}

\begin{proof} 
The fan $\Delta_\sigma$ is a fan over a simplicial sphere and hence has $\dim \CH^d(\Delta) = 1$. Then the space of Minkowski weights is also $1$-dimensional, and it is in fact spanned by a weight that is nonzero on every maximal cone of $\Delta_\sigma$. Thus there is a unique \MW  $u^\sigma$ on $\Delta_\sigma$ such that $u^\sigma_\sigma=w_\sigma$. 

Now consider $u = w-\sum_\sigma u^\sigma \in \MWd(\tilde{\Delta})$. This \MW vanishes on cones $\sigma\in\Delta$. Then it must be zero because the only other maximal cone that contains $\tau\in \Delta_{d-1}$ is $\sigma_\tau$. 
\end{proof}

The decomposition of the Minkowski weight $w = \sum_\sigma u^\sigma$ decomposes the mixed volume 
\[ \Vol_w = \sum_{\sigma \in \Delta_d} \Vol_{u^\sigma}.\]
It turns out that this sum is the same as the sum in Theorem~\ref{thm-MV-eval} and Corollaries~\ref{cor-MV-derivative}-\ref{cor-MV-derivative2}.

\subsection{Mixed volumes of star subdivisions}
\label{sec-MV-star-subdivision}
We will consider here star subdivisions at $2$-dimensional cones.
Let $\Delta$ be a fan in $\RR^n$ and let $\pi: \hat\Delta \to \Delta$ be a star subdivision of $\Delta$ at the cone $\tau$ generated by $v_1,v_2$, with the new ray generated by $v_0 = v_1+v_2$. A \MW $w\in \MWd(\Delta)$ gives rise to a \MW $\hat{w}\in\MWd(\hat\Delta)$ by setting $\hat{w}_{\hat\sigma} = w_{\pi(\hat\sigma)}$ for $\hat\sigma\in \hat\Delta_d$.

Pullback of functions defines an $\RR$-algebra homomorphism $\pi^*: \cA(\Delta)\to \cA(\hat\Delta)$ that induces a morphism $\pi^*: \CH(\Delta) \to \CH(\hat\Delta)$. The morphism of Chow rings is injective, and it is an isomorphism in degree $d$ by Lemma~\ref{lem-star-subdivisions}. The maps $\phi_w$ and $\phi_{\hat{w}}$ agree with the pullback map:
\[ \phi_{\hat{w}} (\pi^*(f)) = \phi_w (f).\]
 
We let $\partial_1,\ldots, \partial_N$ generate $\cA(\Delta)$ and $\hat\partial_0,\ldots, \hat\partial_N$ generate $\cA(\hat\Delta)$. Then pullback $\pi^*$ maps 
\[ \partial_i \mapsto \begin{cases} 
\hat\partial_i + \hat\partial_0 & \text{if $i=1,2$} \\ 
\hat\partial_i & \text{if $i=3,\ldots, N.$}
\end{cases}\]

We first find the mixed volume $\Vol_{\hat{w}}$ on $\hat\Delta$ using the basis $\partial_0=\hat\partial_0,\partial_1,\ldots, \partial_N$ for $\cA^1(\hat\Delta)$. Let the dual variables be $x_0, x_1, x_2, \ldots, x_N$, and write the mixed volume $\Vol_{\hat{w}}$ in these variables as
\[ \Vol_{\hat{w}} = V_0(x_1,\ldots,x_N) + x_0 V_1(x_1, \ldots, x_N) + x_0^2 V_2(x_1, \ldots, x_N) + \cdots.\]
We will later change the variables $x_0,\ldots, x_N$ to variables dual to $\hat\partial_0,\ldots,\hat\partial_N$. 

Since 
\[ \hat\partial_1 \hat\partial_2 = (\partial_1- \partial_0)(\partial_2-\partial_0)\]
is zero in $\cA(\hat\Delta)$, it follows that $\Vol_{\hat{w}}$ must satisfy the differential equation
\begin{equation}
  (\partial_1- \partial_0)(\partial_2-\partial_0) \Vol_{\hat{w}} = 0.
  \label{Vol-subd-de}
\end{equation}
We also have the initial conditions
\begin{equation}
  V_0 = \Vol_w, \quad V_1= 0.
  \label{Vol-subd-ic}
\end{equation}
The first of these follows from the requirement that $\Vol_{\hat{w}}$ applied to the pullback by $\pi$ agrees with $\Vol_w$. The second condition comes from the vanishing of $\partial_0^1 \CH^{d-1}(\Delta)$, see Lemma~\ref{lem-star-subdivisions}.

\begin{lemma}\label{lem-Vol-subd}
  There is a unique solution, \(\Vol_{\hat{w}}(x_0, x_1,\ldots, x_N)\),
  to the differential equation~(\ref{Vol-subd-de}) with initial conditions~(\ref{Vol-subd-ic}):
  \[\Vol_{\hat{w}}(x_0, x_1,\ldots, x_N) = D \Vol_{w}(x_1,\dots,x_N)\]
  where
  \[D = \frac{\partial_1 e^{x_0 \partial_2}  -  \partial_2 e^{x_0 \partial_1}}{\partial_1-\partial_2} =1 - \sum_{m=2}^\infty \frac{x_0^m}{m!} \sum_{\substack{\alpha_1,\alpha_2 \geq 1 \\ \alpha_1+\alpha_2=m}}
  \partial_1^{\alpha_1} \partial_2^{\alpha_2}
  .\]
\end{lemma}
\begin{proof}
  The solution \(\Vol_{\hat{w}}\) is unique since we can solve for \(V_2, V_3, \ldots\) inductively using Equation~(\ref{Vol-subd-de}).
  Now, we show \(D \Vol_{w}\) satisfies the differential equation~(\ref{Vol-subd-de}):
  \[ (\partial_1- \partial_0)(\partial_2-\partial_0) D \Vol_{w} = 0. \]
  We simplify $(\partial_1- \partial_0)(\partial_2-\partial_0) D$ using the commutation relation $\partial_0 x_0 = 1 + x_0 \partial_0$. Since $\Vol_w$ does not involve $x_0$, we may ignore terms in the derivative that end in $\partial_0$. Modulo such terms, we show that \((\partial_1 - \partial_0)(\partial_2 - \partial_0) D = 0\):
  \begin{align*}
    &(\partial_1 - \partial_0)(\partial_2 - \partial_0) D \\
    =& (\partial_1 - \partial_0) (\partial_2 - \partial_0) \left(
    \frac{\partial_1 e^{x_0 \partial_2} - \partial_2 e^{x_0 \partial_1}}{\partial_1-\partial_2} \right) \\
    =& (\partial_1 - \partial_0) \left(
    \frac{\partial_1 \partial_2 e^{x_0 \partial_2} - \partial_2^2 e^{x_0 \partial_1}}{\partial_1-\partial_2} -
    \frac{\partial_1 \partial_2 e^{x_0 \partial_2} - \partial_1
    \partial_2 e^{x_0 \partial_1}}{\partial_1-\partial_2} \right) \\
    =& (\partial_1 - \partial_0) \left( \partial_2 e^{x_0 \partial_1} \right) \\
    =& \partial_1 \partial_2 e^{x_0 \partial_1} - \partial_1 \partial_2 e^{x_0 \partial_1} = 0
  \end{align*}
  Finally, \(D \Vol_{w}\) satisfies the initial conditions~(\ref{Vol-subd-ic}) by definition of \(D\).
\end{proof}

Let us now go back to the variables $\hat\partial_0,\ldots, \hat\partial_N$, with dual variables $\hat{x}_0,\ldots, \hat{x}_N$.
We have
\[ \partial_i = \begin{cases} 
\hat\partial_0 & \text{if $i=0$}\\
\hat\partial_i + \hat\partial_0 & \text{if $i=1,2$} \\ 
\hat\partial_i & \text{if $i=3,\ldots, N$.}
\end{cases}\]
Then the dual variables $x_0,\ldots, x_N$ must satisfy $\partial_i(x_j)=\delta_{i,j}$.  From this we get 
 \[ x_i = \begin{cases} 
\hat{x}_0 - \hat{x}_1- \hat{x}_2 & \text{if $i=0$,}\\
\hat{x}_i & \text{otherwise.}
\end{cases}\]
The mixed volume $\Vol_{\hat{w}}(\hat{x}_0, \hat{x}_1,\ldots, \hat{x}_N)$ is then obtained from $D  \Vol_{w}(x_1,\ldots, x_N)$ by performing the change of variables above:
\[ \Vol_{\hat{w}}(\hat{x}_0, \hat{x}_1,\ldots, \hat{x}_N)
= \left.
    \frac{\partial_1 e^{x_0 \partial_2} - \partial_2 e^{x_0 \partial_1}}{\partial_1-\partial_2} \Vol_{w}(x_1,\ldots, x_N)
\right|_{\substack{
    x_0 = \hat{x}_0 - \hat{x}_1- \hat{x}_2, \\
    x_1= \hat{x}_1 , \dots, x_N = \hat{x}_N}
}.\]
Finally, let us change the names $\hat\partial_i$ to $\partial_i$ and $\hat{x}_i$ to $x_i$. Then

\begin{theorem} \label{thm-MV-star-subd}
The mixed volume $\Vol_{\hat{w}}$ on the star subdivision $\hat{\Delta}$ is 
\[ \Vol_{\hat{w}}(x_0, x_1,\ldots, x_N)
    = \frac{\partial_1 e^{z \partial_2}  -  \partial_2 e^{z \partial_1}}{\partial_1-\partial_2} 
 \Vol_{w}(x_1,\ldots, x_N)|_{z = x_0 - x_1- x_2}.\]
\end{theorem}

\begin{remark}\label{lem-Vol-subd-jgt2}
  One can prove a similar formula in the case of a star subdivision at a cone of dimension $j>2$ with rays generated by $v_1,\dots, v_j$. In this case the derivative $D$ is 
  \[D = 1 - {(-1)}^j \sum_{m=j}^\infty \frac{z^m}{m!}
    \sum_{\substack{\alpha_1,\dots,\alpha_j \geq 1 \\ \alpha_1+\cdots+\alpha_j=m}}
    \partial_1^{\alpha_1} \cdots \partial_j^{\alpha_j},\]
    and the mixed volume is
\[
 \Vol_{\hat{w}}(x_0, x_1,\ldots, x_N)
    =  D \Vol_{w}(x_1,\ldots, x_N)|_{z = x_0 - x_1-\ldots - x_j}.\]
\end{remark}

\section{Matroids}

Let $M$ be a matroid of rank $d+1$ on the set $E=\{1,2,\ldots, n+1\}$ as in the introduction. The Bergman $\Delta_M$ of $M$ is a $d$-dimensional fan embedded in $\RR^n$. The Minkowski weight $(w_\sigma=1)$ spans the space $\MWd$ of $\Delta_M$. We denote the mixed volume corresponding to this weight by $\Vol_M$.

\subsection{Mixed volume  by evaluation}

We prove Theorem~\ref{thm-MV-first-computation} by applying Corollary~\ref{cor-MV-derivative2}. To compute $\Vol_M$, we first choose a general linear projection of the fan $\Delta_M$ to $\RR^d$. This is equivalent to mapping $e_i\mapsto v_i \in \RR^d$ so that the image of $e_E$, which is $\sum_i v_i$, is zero. 

To apply Corollary~\ref{cor-MV-derivative2}, we also need to choose a general vector $v_0\in\RR^d$. The resulting mixed volume does not depend on this vector. Since the projection is general, we may choose $v_0=(0,\ldots,0,1)$. Then the determinant of $A_\sigma$ in Corollary~\ref{cor-MV-derivative2} is equal to a $(d-1)\times(d-1)$ minor in $A_\sigma$. These minors are the determinants in Theorem~\ref{thm-MV-first-computation}.

\subsection{Mixed volume by deletion}

Consider the deletion of $i$ from $M$. As explained in the introduction, we may assume that $i$ is a flat of $M$. Recall the set 
\[ S_i = \{ F_1, F_2,\ldots, F_k\}.\]
Projection from $\Span e_i$ defines a map of Bergman fans $\Delta_M \to \Delta_{M\bs i}$.
We first prove the claim that this map can be factored as
\[ \Delta_M=\Delta_0 \to \Delta_1 \to \cdots \to \Delta_k \to \Delta_{M\bs i},\]
such that for $j=1,\ldots, k$, the map $\Delta_{j-1}\to \Delta_j$ is the star subdivision at the $2$-dimensional cone of $\Delta_j$ with rays corresponding to flats $i,F_j$, and the new ray in $\Delta_{j-1}$ corresponding to the flat $F_j\cup i$. We start by defining the fans $\Delta_j$. 

Let $P$ be a poset. To the poset $P$ is associated a simplicial fan (equivalently, a simplicial complex) where rays correspond to elements in $P$ and cones correspond to chains of these elements in $P$.  Now suppose we have an equivalence relation on $P$, partitioning it into equivalence classes. We define a new simplicial fan in which rays correspond to equivalence classes and cones are images of chains in $P$. To construct the image of a chain, we replace every element in the chain with its equivalence class. This defines an abstract simplicial fan.

Now consider the poset $P=\cLnt(M)$.
Chains in $\cLnt(M)$ correspond to cones in the Bergman fan $\Delta_0=\Delta_M$. Let us first define the fans $\Delta_1,\ldots, \Delta_k$ abstractly, describing a cone as a subset of rays. Let $\Delta_j$ be obtained from the poset $\cLnt(M)$ by declaring $F_l$ equivalent to $F_l \cup i$ for $l=1,\ldots,j$.
 It is then clear that equivalence classes of $\cL(M)$ in the construction of $\Delta_j$ are either singletons or pairs $\{F_l, F_l \cup i\}$ for $l=1,\ldots, j$. Cones in $\Delta_j$ are the images of chains in $\cLnt(M)$. To simplify notation, let us use a flag $\cF$ in $\cLnt(M)$ to denote its image cone in $\Delta_j$.

The map $\Delta_{j-1}\to \Delta_j$ can be viewed as an edge contraction. It contracts the $2$-dimensional cone $F_j < F_j\cup i$ in $\Delta_{j-1}$ to the ray corresponding to the equivalence class $\{F_j, F_j\cup i\}$ in $\Delta_j$. We need to check that this edge contraction is really a star subdivision. Let us start with a lemma.

 \begin{lemma}
For any $F_j \in S_i$, intersection with $F_j$ defines an isomorphism of posets
\[ \cap F_j: \{ G\in \cL(M)| i \leq G \leq F_j\cup i\} \to \{ G\in \cL(M)| \emptyset\leq G\leq F_j\}.\]
\end{lemma}

\begin{proof}
The map is well-defined because intersection of flats is again a flat. Its inverse is defined by $G\mapsto G\cup i$. Let us check that the inverse map is well-defined. If $G\cup i$ is not a flat, then its closure is $G'\cup i$ for some $G\subsetneq G' \subseteq F$. 
This $G' = F_j\cap (G'\cup i)$ is also a flat, hence we have a chain of distinct flats
\[ G \subsetneq G' \subsetneq G'\cup i,\]
which violates the matroid axiom stating that the closure of $G\cup i$ must be a minimal flat above $G$.
\end{proof}

Recall that the flats $F_j\in S_i$ are ordered so that if $F_l\subseteq F_j$ then $l\leq j$. This implies that in $\Delta_{j-1}$ we have declared $G$ equivalent to $G\cup i$ for all $\emptyset < G < F_j$ as in the lemma. In particular, chains in the interval $(i, F\cup i) \subseteq \cL(M)$ map to the same cones in $\Delta_{j-1}$ as chains in the interval $(\emptyset, F) \subseteq \cL(M)$.  To construct $\Delta_j$, we also declare $F_j$ equivalent to $F_j\cup i$.

\begin{proposition} \label{prop-star-subd}
For $j=1,\ldots, k$, the fan $\Delta_{j-1}$ is the star subdivision of $\Delta_j$ at a $2$-dimensional cone $\tau$. The cone $\tau$ has rays $i$ and the equivalence class $\{F_j, F_j\cup i\}$; it is the image in $\Delta_{j}$ of the chain $i<F_j \cup i$ in $\cLnt(M)$.
\end{proposition}

\begin{proof}
As explained in Section~\ref{sec-star-subdivisions}, to prove that $\Delta_{j-1}$ is a star subdivision of $\Delta_j$ with the new ray $F_j\cup i$, we need to show that, as abstract fans,
\[ \Link_{\Delta_{j-1}} F_j\cup i = \Pi^1 \times L,\]
Where $\Pi^1$ is the $1$-dimensional fan with rays $F_j, i$, and $L=\Link_{\Delta_j} \tau$. 

The open star $\Staro{\Delta_{M}}(F_j \cup i)$ consists of all chains in $\cLnt(M)$ that contain $F_j \cup i$, and hence $\Link_{\Delta_{M}}(F_j \cup i)$ consists of all these chains with $F_j \cup i$ removed. 
The images of these chains in $\Delta_{j-1}$ give the open star and the link of $F_j \cup i$ in $\Delta_{j-1}$.

Chains in $\cLnt(M)$ that contain $F_j \cup i$ can be divided into three disjoint sets: those containing $F_j$, $i$, or neither. The previous lemma now implies that the three sets of chains have the same images in $\Delta_{i-1}$ if in the first set we remove $F_j, F_j\cup i$ from each chain, in the second set we remove $i, F_j\cup i$, and in the third set we remove $F_j\cup i$. These images form the subfan $L \subseteq \Delta_{j-1}$. The images of the second set of chains also gives $\Link_{\Delta_j} \tau$ if we view $\tau$ as the image of $i < F_j \cup i$ in $\Delta_j$.
\end{proof}

Let us now define the embedding of the fans $\Delta_j$ for $j=1,\ldots, k$. The Bergman fan $\Delta_M=\Delta_0$ is embedded in $\RR^n$ so that the ray $F$ is generated by $e_F\in\RR^n$. We define the embedding $\Delta_j\hookrightarrow \RR^n$ for $j=1,\dots, k$ so that the ray corresponding to the equivalence class $\{F_l, F_l\cup i\}$ is generated by $e_{F_l} \in \RR^n$, $l=1,\ldots, j$; the other rays corresponding to equivalence classes $\{F\}$ are generated by $e_F$. 

Viewing $\Delta_{j-1}\to \Delta_j$ as an edge contraction, the embeddings of the two fans are such that we contract the $2$-dimensional cone $F_j < F_j\cup i$ to the ray $F_j$. Since $e_{F\cup i} = e_F+e_i$, the fan $\Delta_{j-1}$ is the star subdivision of the fan $\Delta_j$ in $\RR^n$ as considered above. Note that since the projection from $e_i$ maps $e_{F\cup i}$ and $e_{F}$ to the same vector, the projection map $\Delta_M\to \Delta_{M\bs i}$ factors through $\Delta_j$:
\[ \Delta_M=\Delta_0 \to \Delta_1 \to \cdots \to \Delta_k\to \Delta_{M\bs i}.\]

\begin{remark}
The proof of the proposition shows that $\Link_{\Delta_M} (F_j < F_j\cup i)$ and $\Link_{\Delta_j} \tau$ are the same fans. The first of these is the product of two fans, namely the Bergman fans of the matroids $M^{F_j}$ and $M_{F_j\cup i}$. The embedding of the link is a product embedding. The link of $\tau$ in $\Delta_j$ is the same product fan with the same product embedding. This can be used to deduce Hard Lefschetz theorem and Hodge Riemann bilinear relations for the link of $\tau$ using induction on dimension.
\end{remark}

We next consider how to compute the mixed volume of $\Delta_{j-1}$ from the mixed volume of $\Delta_j$. 
Theorem~\ref{thm-MV-star-subdiv} in Section~\ref{sec-MV-star-subdivision} describes this for a star subdivision. It remains to consider the last step in the factorization $\Delta_k \to \Delta_{M\bs i}$ and prove Theorems~\ref{MV-non-coloop}-\ref{MV-coloop}.

Let $\pi:\Delta_M \to \Delta_{M\bs i}$ be the map of fans defined by projection from $e_i$. Let $\pi^*: \cA(\Delta_{M\bs i}) \to \cA(\Delta_M)$ be the $\RR$-algebra homomorphism given by pullback of functions.

\begin{lemma}
The map $\pi^*$ induces a homomorphism of $\RR$-algebras
\[ \pi^*: \CH(M\bs i) \to \CH(M).\]
The homomorphism is compatible with the degree maps in the following sense. If $i$ is not a coloop, then 
\[ deg_M\circ\pi^* = \deg_{M\bs i}.\]
If $i$ is a coloop, then 
\[ deg_M\circ (\partial_i \cdot \pi^*) = deg_M\circ (\partial_{E\bs i} \cdot \pi^*) =\deg_{M\bs i}.\]
In particular, if $i$ is not a coloop, then $\pi^*$ is injective. If $i$ is a coloop then $\partial_i \cdot \pi^*$ and $\partial_{E\bs i} \cdot \pi^*$ are injective.
\end{lemma}

\begin{proof}
Since global linear functions pull back to global linear functions, $\pi^*$ induces an $\RR$-algebra homomorphism of Chow rings as stated. 

When $i$ is not a coloop, then $\pi$ is a surjective map of fans of the same dimension, taking cones onto cones. If a maximal cone $\sigma$ maps onto a maximal cone $\pi(\sigma)$, then generators of $\sigma$ map to generators of $\pi(\sigma)$. Hence $\partial_{\pi(\sigma)} \in \CH(M\bs i)$ pulls back to $\partial_{\sigma} \in \CH(M)$. Since these monomials have degree $1$, we get the compatibility of the degree maps. 

When $i$ is not a coloop, then $\pi$ reduces dimension by $1$. A maximal cone  $\sigma\in \Delta_ M$ that contains the ray $i$ maps onto a maximal cone in $\Delta_{M\bs i}$. Then $\partial_\sigma = \partial_i \pi^*(\partial_{\pi(\sigma)})$. The degree compatibility follows from this. Similarly for maximal cones containing the ray $E\bs i$.

\Po duality now implies the injectivity of the maps as stated. Let us prove this for  $\partial_i \cdot \pi^*$. If $h\in \CH(M\bs i)$ is a nonzero element, then $deg(hg)\neq 0$ for some $g\in \CH(M\bs i)$. Now  $deg(hg) = deg(\partial_i \pi^*(h) \pi^*(g)) \neq 0$, which implies that $\partial_i \pi^*(h) \neq 0$.
\end{proof}

Consider now the factorization of the map $\pi$:
\[ \Delta_M \stackrel{\pi_1}{ \longrightarrow} \Delta_k \stackrel {\pi_2}{\longrightarrow} \Delta_{M\bs i}.\]
The first map $\pi_1$ is a sequence of star subdivisions. By Lemma~\ref{lem-star-subdivisions}, pullback induces an isomorphims
\[ \pi_1^*: \CH^d(\Delta_k) \isomto \CH^d(M).\]
In fact,  the degree $d$ constant \MW $(w_\sigma=1)$ on $\Delta_M$ maps by the dual of this isomorphism to the degree $d$ constant \MW $(w_\sigma=1)$ on $\Delta_k$. In the previous lemma we may now replace $\pi: \Delta_M \to \Delta_{M\bs i}$ with $\pi_2: \Delta_k \to \Delta_{M\bs i}$. Then the statement and the proof of the lemma apply word-by-word. 

Let us now compute the mixed volume of $\Delta_k$ from the mixed volume of $\Delta_{M\bs i}$. Both mixed volumes are defined using the constant Minkowski weight $(w_\sigma=1)$ on maximal cones. 
First assume that $i$ is not a coloop. The element
\begin{equation} \label{eq-non-coloop}
 \partial_i - \sum_{G\in \cLnt(M\bs i)} b_{\overline{G}} \partial_{\overline{G}}
\end{equation}
vanishes in $\CH(\Delta_k)$. We consider this as a derivative that must annihilate the mixed volume of $\Delta_k$. In addition, we know how the mixed volume acts on the image of  $\pi_2^*: \CH(M\bs i) \hookrightarrow \CH(\Delta_k)$, where the inclusion map is defined by $\partial_G \mapsto \partial_{\overline{G}}$.

\begin{lemma}
Let 
\[ \Vol_{M\bs i}(x_{{G}}) \in \RR[x_{{G}}]_{G\in\cLnt(M\bs i)}.\]
Then there exists a unique polynomial 
\[ \Vol_{\Delta_k} = \Vol_{M\bs i}( x_{\overline{G}} + b_{\overline{G}} x_i) \in \RR[ x_i, x_{\overline{G}}]_{G\in\cLnt(M\bs i)}\]
that is annihilated by the derivative (\ref{eq-non-coloop}) and that satisfies the initial condition
\[ \Vol_{\Delta_k}|_{x_i=0} = \Vol_{M\bs i}(x_{\overline{G}}).\]
\end{lemma}

\begin{proof}
Clearly there is a unique such polynomial because we can solve for the coefficients of $x_i, x_i^2,\ldots$ inductively using vanishing of the derivative (\ref{eq-non-coloop}). Let us check that the given polynomial satisfies the differential equation. Applying the chain rule gives
\[ \partial_i \Vol_{M\bs i}( x_{\overline{G}} + b_{\overline{G}} x_i) = \sum_{G\in \cLnt(M\bs i)} b_{\overline{G}} \partial_G \Vol_{M\bs i}( x_{\overline{G}} + b_{\overline{G}} x_i).\]
Again by chain rule, this is the same as 
\[\sum_{G\in \cLnt(M\bs i)} b_{\overline{G}} \partial_{\overline{G}} \Vol_{M\bs i}( x_{\overline{G}} + b_{\overline{G}} x_i).\]
Since the initial condition clearly holds, this finishes the proof of the lemma and Theorem~\ref{MV-non-coloop}.
\end{proof}

Next assume that $i$ is a coloop. Similarly to the case above, we now have a derivative
\begin{equation} \label{eq-coloop}
 \partial_i -\partial_{E\bs i} - \sum_{G\in \cLnt(M\bs i)} b_{\overline{G}} \partial_{\overline{G}}
\end{equation}
that must annihilate the mixed volume of $\Delta_k$. The relation $\partial_i\partial_{E\bs i} = 0$ in $\CH(\Delta_k)$ gives another such derivative. 

\begin{lemma}
Let 
\[ \Vol_{M\bs i}(x_{{G}}) \in \RR[x_{{G}}]_{G\in\cLnt(M\bs i)}.\]
Then there exists a unique polynomial $\Vol_{\Delta_k} \in \RR[ x_i, x_{E\bs i}, x_{\overline{G}}]_{G\in\cLnt(M\bs i)}$,
\[ \Vol_{\Delta_k} = \int \Vol_{M\bs i}(x_{\overline{G}} + b_{\overline{G}} x_i) dx_i + \int \Vol_{M\bs i} (x_{\overline{G}} - b_{\overline{G}} x_{E\bs i}) dx_{E\bs i}\]
that is annihilated by the derivatives (\ref{eq-coloop}) and $\partial_i\partial_{E\bs i}$, and that satisfies the initial condition
\[ \Vol_{\Delta_k} = (x_i+x_{E\bs i}) \Vol_{M\bs i}(x_{\overline{G}}) + \text{terms of degree $\geq 2$ in $x_i, x_{E\bs i}$}.\]
\end{lemma}

\begin{proof}
Let $V_1 = \Vol_{M\bs i}(x_{\overline{G}} + b_{\overline{G}} x_i)$ and $V_2 = \Vol_{M\bs i} (x_{\overline{G}} - b_{\overline{G}} x_{E\bs i})$. It suffices to prove that applying the derivative (\ref{eq-non-coloop}) to the term with $V_1$ gives 
\[ (\partial_i - \sum_{G\in \cLnt(M\bs i)} b_{\overline{G}} \partial_{\overline{G}}) \int V_1 dx_i = -\Vol_{M\bs i}(x_{\overline{G}}),\]
and applying a similar derivative to the term  with $V_2$ gives
\[ (-\partial_{E\bs i} - \sum_{G\in \cLnt(M\bs i)} b_{\overline{G}} \partial_{\overline{G}}) \int V_2 dx_{E\bs i} = \Vol_{M\bs i}(x_{\overline{G}}).\]
Then (\ref{eq-coloop}) applied to $\Vol_{\Delta_k}$ vanishes. Other statements of the lemma are clear.

Let us prove the first equality. The second one is similar. We assume that the derivative 
(\ref{eq-non-coloop}) applied to $V_1$ vanishes. Then
\begin{align*}
    (\partial_i - \sum_{G\in \cLnt(M\bs i)} b_{\overline{G}} \partial_{\overline{G}}) \int V_1 dx_i &= V_1 - \int \sum_{G\in \cLnt(M\bs i)} b_{\overline{G}} \partial_{\overline{G}} V_1 dx_i \\
    &= V_1 - \int \partial_i V_1 dx_i \\
    &= V_1|_{x_1=0}.
\end{align*}
This finishes the proof of the lemma and Theorem~\ref{MV-coloop}.
\end{proof}

\bibliographystyle{plain}
\bibliography{mixedvol}{}

\begin{thebibliography}{10}

\bibitem{AHK}
Karim Adiprasito, June Huh, and Eric Katz.
\newblock Hodge theory for combinatorial geometries.
\newblock {\em Ann. of Math. (2)}, 188(2):381--452, 2018.

\bibitem{BES}
Spencer Backman, Christopher Eur, and Connor Simpson.
\newblock Simplicial generation of {C}how rings of matroids.
\newblock {\em S\'{e}m. Lothar. Combin.}, 84B:Art. 52, 11, 2020.

\bibitem{BBFK}
Gottfried Barthel, Jean-Paul Brasselet, Karl-Heinz Fieseler, and Ludger Kaup.
\newblock Combinatorial intersection cohomology for fans.
\newblock {\em Tohoku Math. J. (2)}, 54(1):1--41, 2002.

\bibitem{BHM}
Tom Braden, June Huh, Jacob~P. Matherne, Nicholas Proudfoot, and Botong Wang.
\newblock A semi-small decomposition of the {C}how ring of a matroid.
\newblock {\em Adv. Math.}, 409:Paper No. 108646, 49, 2022.

\bibitem{HuhBranden}
Petter Br\"{a}nd\'{e}n and June Huh.
\newblock Lorentzian polynomials.
\newblock {\em Ann. of Math. (2)}, 192(3):821--891, 2020.

\bibitem{BL1}
Paul Bressler and Valery~A. Lunts.
\newblock Intersection cohomology on nonrational polytopes.
\newblock {\em Compositio Math.}, 135(3):245--278, 2003.

\bibitem{Brion}
Michel Brion.
\newblock The structure of the polytope algebra.
\newblock {\em Tohoku Math. J. (2)}, 49(1):1--32, 1997.

\bibitem{Eur}
Christopher Eur.
\newblock Divisors on matroids and their volumes.
\newblock {\em J. Combin. Theory Ser. A}, 169:105135, 31, 2020.

\bibitem{FeichtnerYuzvinsky}
Eva~Maria Feichtner and Sergey Yuzvinsky.
\newblock Chow rings of toric varieties defined by atomic lattices.
\newblock {\em Invent. Math.}, 155(3):515--536, 2004.

\bibitem{FultonSturmfels}
William Fulton and Bernd Sturmfels.
\newblock Intersection theory on toric varieties.
\newblock {\em Topology}, 36(2):335--353, 1997.

\bibitem{KX}
Kalle Karu and Elizabeth Xiao.
\newblock On the anisotropy theorem of {P}apadakis and {P}etrotou.
\newblock {\em Algebr. Comb.}, 6(5):1313--1330, 2023.

\bibitem{NOR}
Lauren Nowak, Patrick O'Melveny, and Dustin Ross.
\newblock Mixed volumes of normal complexes.
\newblock {\em arXiv:2301.05278}, 2023.

\bibitem{Oxley}
James Oxley.
\newblock {\em Matroid theory}, volume~21 of {\em Oxford Graduate Texts in
  Mathematics}.
\newblock Oxford University Press, Oxford, second edition, 2011.

\end{thebibliography}

\end{document}